\documentclass[a4paper,10pt]{amsart}

\usepackage{amsmath,amsfonts,amssymb}
\usepackage[all]{xy}
\usepackage[english]{babel}
\usepackage[utf8x]{inputenc}
\usepackage{graphicx}
\usepackage{multirow}
\usepackage{array}
\usepackage{booktabs}
\usepackage{ifthen}
\usepackage{color}

\usepackage{fullpage}

\newtheorem{prop}{Proposition}[section]
\newtheorem{thm}[prop]{Theorem}
\newtheorem{cor}[prop]{Corollary}
\newtheorem{lemma}[prop]{Lemma}
\theoremstyle{definition}

\newtheorem{rem}[prop]{Remark}
\newtheorem{ex}[prop]{Example}

\newcommand{\kth}[1]{\ifthenelse{\equal{#1}{1}}{$#1^\text{st}$}{\ifthenelse{\equal{#1}{2}}{$#1^\text{nd}$}{\ifthenelse{\equal{#1}{3}}{$#1^\text{rd}$}{$#1^\text{th}$}}}}

\newcommand{\N}{\mathbb{N}}

\newcommand{\Q}{\mathbb{Q}}
\newcommand{\R}{\mathbb{R}}
\newcommand{\C}{\mathbb{C}}

\newcommand{\st}{\;:\;}

\newcommand{\correnti}{\mathcal{D}}

\newcommand{\sspace}{\cdot}
\newcommand{\ssspace}{\cdot\cdot}

\newcommand{\duale}[1] {{{#1}^{*}}}

\newcommand{\Prim}{{\mathrm{P}\wedge}}
\newcommand{\Primcorrenti}{{\mathrm{P}\correnti}}
\newcommand{\PH}{\mathrm{P}H}

\DeclareMathOperator{\im}{i}
\DeclareMathOperator{\imm}{im}

\DeclareMathOperator{\de}{d}
\DeclareMathOperator{\id}{id}
\DeclareMathOperator{\Vol}{Vol}

\newcommand{\del}{\partial}
\newcommand{\delbar}{\overline{\del}}

\title[Symplectic manifolds and cohomological decomposition]{Symplectic manifolds and cohomological decomposition}

\author{Daniele Angella}
\address[Daniele Angella]{Dipartimento di Matematica\\
Universit\`{a} di Pisa \\
Largo Bruno Pontecorvo 5, 56127\\
Pisa, Italy}
\email{angella@mail.dm.unipi.it}

\author{Adriano Tomassini}
\address[Adriano Tomassini]{Dipartimento Di Matematica e Informatica\\
Universit\`{a} di Parma \\
Parco Area delle Scienze 53/A, 43124 \\
Parma, Italy}
\email{adriano.tomassini@unipr.it}

\thanks{This work was supported by GNSAGA of INdAM}

\begin{document}

\begin{abstract}
 Given a closed symplectic manifold, we study when the Lefschetz decomposition induced by the $\mathfrak{sl}(2;\R)$-representation yields a decomposition of the de Rham cohomology. In particular, this holds always true for the second de Rham cohomology group, or if the symplectic manifold satisfies the Hard Lefschetz Condition.
\end{abstract}

\maketitle

\section*{Introduction}

Compact K\"ahler manifolds have special cohomological properties: from the complex point of view, the Hodge decomposition theorem states that the complex de Rham cohomology groups decompose as direct sum of the Dolbeault cohomology groups, and from the symplectic side the Hard Lefschetz theorem provides a decomposition of the de Rham cohomology as direct sum of primitive cohomology groups. Such decompositions do not hold anymore for general non-K\"ahler complex manifolds.

\medskip

To the purpose of generalizing the above cohomological complex-type decomposition on an arbitrary almost-complex manifold $\left(X,\, J\right)$, T.-J. Li and W. Zhang have introduced in \cite{li-zhang} the subgroups $H^{(p,q),(q,p)}_J(X;\R)\subseteq H^\bullet_{dR}(X;\R)$ (respectively, $H^{(p,q)}_J(X)\subseteq H^\bullet_{dR}(X;\C)$), formed by the real (respectively, complex) de Rham cohomology classes having representatives of pure type $(p,q)$ (we refer also to \cite{angella-tomassini, angella-tomassini-2, fino-tomassini} for further results concerning these subgroups). In \cite{draghici-li-zhang}, T. Dr\v{a}ghici, T.-J. Li and W. Zhang have proved that any closed $4$-dimensional manifold endowed with an almost-complex structure $J$ satisfies the decomposition $H^2_{dR}(X;\R) = H^{(1,1)}_J(X;\R) \oplus H^{(2,0),(0,2)}_J(X;\R)$, which can be regarded as a Hodge decomposition for non-K\"ahler $4$-manifolds. This decomposition does not hold true in higher dimension, see \cite[Example 3.3]{fino-tomassini}.

In \cite{brylinski}, J.-L. Brylinski proposed a Hodge theory for closed symplectic manifolds $\left(X,\,\omega\right)$: in this context, O. Mathieu in \cite{mathieu}, and D. Yan in \cite{yan}, proved that any de Rham cohomology class admits a symplectically harmonic representative (i.e., a form being both $\de$-closed and $\de^\Lambda$-closed, where $\de^\Lambda\lfloor_{\wedge^kX}:=(-1)^{k+1}\,\star_\omega\de\star_\omega$, and $\star_\omega$ is the symplectic-$\star$-operator) if and only if the Hard Lefschetz Condition is satisfied.\\
Recently, L.-S. Tseng and S.-T. Yau, in \cite{tseng-yau-1, tseng-yau-2} (see also \cite{tseng-yau-3}), have introduced new cohomologies for symplectic manifolds $\left(X,\,\omega\right)$: among them, in particular, they have defined and studied
$$ H^\bullet_{\de+\de^\Lambda}(X;\R) \;:=\; \frac{\ker\left(\de+\de^\Lambda\right)}{\imm\de\de^\Lambda} \;,$$
developing a Hodge theory for such cohomology; $H^\bullet_{\de+\de^\Lambda}(X;\R)$ can be interpreted as the symplectic counterpart to the Bott-Chern cohomology of a complex manifold, see \cite{tseng-yau-3}. (As regards the Bott-Chern cohomology and its relation with the cohomological properties of a compact complex manifold, we refer to \cite{angella-tomassini-3}, where the problem whether the Bott-Chern cohomology groups induce a decomposition of the de Rham cohomology is studied, and a characterization of compact complex manifolds satisfying the $\del\delbar$-Lemma is given.)
Furthermore, they have studied the dual currents of Lagrangian and co-isotropic submanifolds, and they have defined a homology theory on co-isotropic chains, which turns out to be naturally dual to a primitive cohomology.\\
Inspired also by their work, Y. Lin has developed in \cite{lin} a new approach to the symplectic Hodge theory, proving in particular that, on any closed symplectic manifold satisfying the Hard Lefschetz Condition, there is a Poincaré duality between the primitive homology on co-isotropic chains and the primitive cohomology.

\medskip

In the present paper, we focus on cohomological properties of closed symplectic manifolds $\left(X,\,\omega\right)$. Denote by
$$ H^{(r,s)}_\omega(X;\R) \;:=\; \left\{\left[L^r\,\beta^{(s)}\right]\in H^{2r+s}_{dR}(X;\R) \st \beta^{(s)} \text{ is a primitive }s\text{-form} \right\} \;\subseteq\; H^\bullet_{dR}(X;\R) \;, $$
and by
$$ H_{(r,s)}^\omega(X;\R) \;:=\; \left\{\left[L^r\,\beta_{(s)}\right]\in H_{-2r+s}^{dR}(X;\R) \st \beta_{(s)} \text{ is a primitive }s\text{-current} \right\} \;\subseteq\; H_\bullet^{dR}(X;\R) \;, $$
where $L\colon \wedge^\bullet X \to \wedge^{\bullet+2}X$ is defined by $L\alpha:=\omega\wedge\alpha$, and $L\colon \correnti_\bullet X \to \correnti_{\bullet-2}X$ is induced by duality, where $\correnti_\bullet X$ denotes the complex of currents on $X$.\\
We are concerned in studying when the above subgroups yield a direct sum decomposition of the de Rham cohomology, respectively, of the de Rham homology. In this matter, we prove the following result, which can be regarded as the symplectic counterpart to \cite[Theorem 2.3]{draghici-li-zhang} by T. Dr\v{a}ghici, T.-J. Li and W. Zhang in the complex case.

\smallskip
\noindent {\bfseries Theorem \ref{thm:sympl-decomp-H2}.\ }{\itshape
 Let $X$ be a closed manifold endowed with a symplectic structure $\omega$. Then
 \begin{eqnarray*}
  H^2_{dR}(X;\R) &=& H^{(1,0)}_\omega(X;\R) \oplus H^{(0,2)}_\omega(X;\R) \;.
 \end{eqnarray*}
 In particular, if $\dim X=4$, then
 $$ H^{\bullet}_{dR}(X;\R) \;=\; \bigoplus_{r\in\N} H^{(r,\bullet-2r)}_\omega(X;\R) $$
 and
 $$ H_{\bullet}^{dR}(X;\R) \;=\; \bigoplus_{r\in\N} H_{(r,\bullet+2r)}^\omega(X;\R) \;.$$
}
\smallskip

If $\left( X, \, \omega \right)$ satisfies the Hard Lefschetz Condition, then $H^\bullet_{dR}(X;\R)$, respectively $H_\bullet^{dR}(X;\R)$, decomposes as direct sum of the subgroups $H^{(\bullet,\bullet)}_{\omega}(X;\R)$, respectively $H_{(\bullet,\bullet)}^{\omega}(X;\R)$, see Corollary \ref{cor:hlc}.\\
Then we specialize on \emph{solvmanifolds}, namely, compact quotients of solvable Lie groups, showing a Nomizu theorem for solvmanifolds of completely-solvable type, see Proposition \ref{prop:linear-cpf-invariant-cpf}.

\medskip

The paper is organized as follows. In Section \ref{sec:preliminaries}, we recall the basic facts concerning Lefschetz decomposition (both for differential forms, for currents, and for cohomologies) and Hodge theory on closed symplectic manifolds, in particular with the aim to fix the notations. In Section \ref{sec:cohomology}, we introduce and study the subgroups $H^{(\bullet,\bullet)}_\omega(X;\R)$, proving the symplectic cohomological decomposition in Theorem \ref{thm:sympl-decomp-H2}. In Section \ref{sec:solvmanifolds}, we study symplectic cohomology decomposition for solvmanifolds, providing several explicit examples and computing the symplectic cohomology groups in such cases.

\medskip

\noindent{\sl Acknowledgments.} The authors would like to thank Yi Lin for pointing them the reference \cite{lin}, and for useful suggestions.

\section{Preliminaries on Hodge theory for symplectic manifolds}\label{sec:preliminaries}

We recall here some notions and results concerning Hodge theory for symplectic manifolds, referring to \cite{brylinski, mathieu, yan, cavalcanti, tseng-yau-1, tseng-yau-2, lin}.

\subsection{Primitive forms and Lefschetz decomposition}
Let $\left(V,\, \omega\right)$ be a $2n$-dimensional symplectic vector space and denote by $\{e_1,\ldots ,e_{2n}\}$ a Darboux basis of $V$ for $\omega$, i.e., $\omega =\sum_{i=1}^n e^i\wedge e^{n+i}$, where $\{e^1,\ldots ,e^{2n}\}$ is the dual basis of $\{e_1,\ldots ,e_{2n}\}$. Denote by $I:V\to V^*$ the natural isomorphism induced by $\omega$, namely $I(v)(\cdot)=\omega(v,\cdot)$, for every $v\in V$. Then $\omega$ gives rise to a bilinear form on $\wedge^k \duale{V}$, denoted by $\left(\omega^{-1}\right)^k$, which is skew-symmetric, respectively symmetric, according that $k$ is odd, respectively even, and defined on the simple elements as
$$
\left(\omega^{-1}\right)^k\left(\alpha^1\wedge\ldots \wedge \alpha^k,\beta^1\wedge\ldots \wedge \beta^k\right) \;:=\; \det\left(\omega^{-1}\left(\alpha^i,\beta^j\right)\right)_{i,j\in\{1,\ldots,k\}} \;,
$$
where $\omega^{-1}\left(\alpha^i,\beta^j\right) \;:=\; \omega\left(I^{-1}\left(\alpha^i\right),I^{-1}\left(\beta^j\right)\right)$. The \emph{symplectic-$\star$-operator}
$$
\star_\omega\colon \wedge^\bullet X\to \wedge^{2n-\bullet}X
$$
is defined requiring that, for every $\alpha,\,\beta\in\wedge^k\duale{V}$,
$$ \alpha\wedge\star_\omega \beta \;=\; \left(\omega^{-1}\right)^k\left(\alpha,\beta\right)\,\omega^n \;,$$
see \cite[\S2]{brylinski}.

\medskip

Let $X$ be a $2n$-dimensional closed manifold and let $\omega$ be a \emph{symplectic} structure on $X$ (namely, a non-degenerate $\de$-closed $2$-form on $X$). Set $\Pi:=\omega^{-1}\in\wedge^2TX$ the canonical Poisson bi-vector associated to $\omega$, namely, in a Darboux chart with local coordinates $\left\{x^1,\ldots,x^n,y^1,\ldots,y^n\right\}$, if $\omega\stackrel{\text{loc}}{=}\sum_{j=1}^{n}\de x^j\wedge \de y^j$, then $\omega^{-1}\stackrel{\text{loc}}{=}\sum_{j=1}^{n}\frac{\del}{\del x^j}\wedge\frac{\del}{\del y^j}$. Consider the $\mathfrak{sl}(2;\R)$-representation on $\wedge^\bullet X$ given by $\left\langle L,\, \Lambda,\, H \right\rangle\subset \mathrm{End}^\bullet\left(\wedge^\bullet X\right)$, where
\begin{eqnarray*}
L\colon \wedge^\bullet X\to \wedge^{\bullet+2}X\;, \quad && \alpha\mapsto \omega\wedge\alpha \;,\\[5pt]
\Lambda\colon \wedge^\bullet X\to \wedge^{\bullet-2}X\;, \quad && \alpha \mapsto -\iota_\Pi\alpha \;,\\[5pt]
H\colon \wedge^\bullet X\to \wedge^\bullet X\;, \quad && \alpha\mapsto \sum_k \left(n-k\right)\,\pi_{\wedge^kX}\alpha \;,
\end{eqnarray*}
(we denote the interior product with $\xi\in\wedge^2\left(TX\right)$ by $\iota_{\xi}\colon\wedge^{\bullet}X\to\wedge^{\bullet-2}X$, and, for $k\in\N$, the natural projection by $\pi_{\wedge^kX}\colon\wedge^\bullet X\to\wedge^kX$). Using the symplectic-$\star$-operator $\star_\omega$, one can write
$$
\Lambda \;=\; -\star_\omega\,L\,\star_\omega \;.
$$
The above $\mathfrak{sl}(2;\R)$-representation, having finite $H$-spectrum, induces the \emph{Lefschetz decomposition} on differential forms, \cite[Corollary 2.6]{yan},
$$ \wedge^\bullet X \;=\; \bigoplus_{r\in\N} L^r \, \Prim^{\bullet-2r}X \;, $$
where
$$ \Prim^\bullet X \;:=\; \ker\Lambda $$
is the space of \emph{primitive forms}. Note (see, e.g., \cite[Proposition 1.2.30(v)]{huybrechts}) that, for every $k\in\N$,
$$ \Prim^kX \;=\; \ker L^{n-k+1}\lfloor_{\wedge^kX} \;.$$
In general, see \cite[pages 7-8]{tseng-yau-2}, the Lefschetz decomposition of $A^{(k)}\in\wedge^kX$ reads as
$$ A^{(k)} \;=\; \sum_{r\geq \max\left\{k-n,\,0\right\}}\frac{1}{r!}\, L^r\, B^{(k-2r)} $$
where, for $r\geq\max\left\{k-n,\,0\right\}$,
$$ B^{(k-2r)} \;:=\; \left(\sum_{\ell\in\N}a_{r,\ell,(n,k)}\,\frac{1}{\ell!}\, L^\ell\, \Lambda^{r+\ell}\right)\, A^{(k)} \;\in\; \Prim^{k-2r}X $$
and, for $r\geq\max\left\{k-n,\,0\right\}$ and $\ell\in\N$,
$$ a_{r,\ell,(n,k)} \;:=\; \left(-1\right)^{\ell} \cdot \left(n-k+2r+1\right)^2 \cdot \prod_{i=0}^{r}\frac{1}{n-k+2r+1-i} \cdot \prod_{j=0}^{\ell}\frac{1}{n-k+2r+1+j} \in \Q \;. $$

\medskip

We recall that
$$ L\lfloor_{\bigoplus_{k=-1}^{n-2}\wedge^{n-k-2}X} \colon \bigoplus_{k=-1}^{n-2}\wedge^{n-k-2}X \to \wedge^{n-k}X $$
is injective, see \cite[Corollary 2.8]{yan}, and that, for every $k\in\N$,
$$ L^k\colon \wedge^{n-k}X\to\wedge^{n+k}X $$
is an isomorphism, see \cite[Corollary 2.7]{yan}.

\subsection{Symplectic cohomologies}
Set
$$
 \de^\Lambda\lfloor_{\wedge^kX} \;:=\; (-1)^{k+1}\star_\omega \de \star_\omega
$$
for every $k\in \N$. Then the following basic symplectic identity holds (see, e.g., \cite[Corollary 1.3]{yan}):
\begin{equation}\label{basic-symplectic-identity}
\left[\de,\,\Lambda\right]\;=\;\de^\Lambda \;.
\end{equation}
As a direct consequence of \eqref{basic-symplectic-identity}, one gets $\de\de^\Lambda +\de^\Lambda\de =0$.\newline
Note also that, if $\left(J,\,\omega,\,g\right)$ is an almost-K\"ahler structure on $X$, then the symplectic-$\star$-operator $\star_\omega$ and the Hodge-$*$-operator $*_g$ are related by $\star_\omega=J\,*_g$, and hence $\de^\Lambda=-\left(\de^c\right)^{*_g}$ where $\de^c:=-\im\left(\del-\delbar\right)$.)

\medskip

Being $\left(\de^\Lambda\right)^2=0$, one can consider, as in \cite[\S1]{brylinski} and \cite[\S3.1]{tseng-yau-1}, the following cohomology:
$$ H^\bullet_{\de^\Lambda}(X;\R) \;:=\; \frac{\ker \de^\Lambda}{\imm\de^\Lambda} \;, $$
which is isomorphic to the de Rham cohomology, since
$$ \star_\omega\colon H^\bullet_{dR}(X;\R) \stackrel{\simeq}{\longrightarrow} H^{2n-\bullet}_{\de^\Lambda}(X;\R) \;,$$
by \cite[Corollary 2.2.2]{brylinski}.

\medskip

In \cite{tseng-yau-1}, looking for a symplectic counterpart to the Aeppli and Bott-Chern cohomologies of complex manifolds (see \cite{tseng-yau-3} for further discussions), L.-S. Tseng and S.-T. Yau introduce also the \emph{$\left(\de+\de^\Lambda\right)$-cohomology}, \cite[\S3.2]{tseng-yau-1},
$$ H^\bullet_{\de+\de^\Lambda} \;:=\; \frac{\ker\left(\de+\de^\Lambda\right)}{\imm\de\de^\Lambda} \;, $$
and the \emph{$\left(\de\de^\Lambda\right)$-cohomology}, \cite[\S3.3]{tseng-yau-1},
$$ H^\bullet_{\de\de^\Lambda} \;:=\; \frac{\ker \de\de^\Lambda}{\imm\de+\imm\de^\Lambda} \;, $$
proving in \cite[Corollary 3.6, Corollary 3.17]{tseng-yau-1} that, being $X$ compact, they are finite-dimensional $\R$-vector spaces, since, once fixed an almost-K\"ahler structure $\left(J,\, \omega,\, g\right)$ on $X$, they are isomorphic to the kernel of certain \kth{4}-order self-adjoint elliptic differential operators, see \cite[Theorem 3.5, Theorem 3.16]{tseng-yau-1};
furthermore, the Hodge-$*$-operator with respect to $g$ induces
$$ *\colon H^\bullet_{\de+\de^\Lambda}(X;\R) \stackrel{\simeq}{\longrightarrow} H^{2n-\bullet}_{\de\de^\Lambda}(X;\R) \;,$$
see \cite[Corollary 3.25]{tseng-yau-1}.\\
Moreover, it is proven in \cite[Proposition 2.8]{tseng-yau-1} that the cohomology $H^\bullet_{\de+\de^\Lambda}(X;\R)$ is invariant under symplectomorphism and Hamiltonian isotopy.

\medskip

The following commutation relations between the differential operators $\de$, $\de^\Lambda$, and $\de\de^\Lambda$, and the $\mathfrak{sl}(2;\R)$-module generators $L$, $\Lambda$, and $H$, hold straightforwardly, \cite[Lemma 2.3]{tseng-yau-1}:
$$
\begin{array}{rclrclrcl}
\left[\de,\,L\right] &=& 0 \;, & \left[\de^\Lambda,\,L\right] &=& -\de \;, & \left[\de\de^\Lambda,\,L\right] &=& 0 \;, \\[5pt]
\left[\de,\,\Lambda\right] &=:& \de^\Lambda \;, & \left[\de^\Lambda,\,\Lambda\right] &=& 0 \;, & \left[\de\de^\Lambda,\,\Lambda\right] &=& 0 \;,\\[5pt]
\left[\de,\,H\right] &=& \de \;, & \left[\de^\Lambda,\,H\right] &=& -\de^\Lambda \;, & \left[\de\de^\Lambda,\,H\right] &=& 0 \;.
\end{array}
$$
Hence, setting
$$ \PH^\bullet_{\de+\de^\Lambda}(X;\R) \;:=\; \frac{\ker\de\cap\ker\de^\Lambda\cap\Prim^\bullet X}{\imm\de\de^\Lambda\cap\Prim^\bullet X} \;=\; \frac{\ker\de\cap\Prim^\bullet X}{\de\de^\Lambda\left(\Prim^\bullet X\right)} $$
(see \cite[Lemma 3.9]{tseng-yau-1}), one gets that
$$ H^\bullet_{\de+\de^\Lambda}(X;\R) \;=\; \bigoplus_{r\in\N} L^r \, \PH^{\bullet-2r}_{\de+\de^\Lambda}(X;\R) $$
and, for every $k\in\N$,
$$ L^k\colon H^{n-k}_{\de+\de^\Lambda}(X;\R) \stackrel{\simeq}{\longrightarrow} H^{n+k}_{\de+\de^\Lambda}(X;\R) \;,$$
see \cite[Theorem 3.11]{tseng-yau-1}.

\medskip

\subsection{Hard Lefschetz condition}
The identity map induces the following natural maps in cohomology:
$$
\xymatrix{
  & H_{\de+\de^\Lambda}^{\bullet}(X;\R) \ar[ld]\ar[rd] &   \\
 H_{dR}^{\bullet}(X;\R) \ar[rd] &  & H_{\de^\Lambda}^{\bullet}(X;\R) \ar[ld] \\
  & H_{\de\de^\Lambda}^{\bullet}(X;\R) &
}
$$

Recall that a symplectic manifold is said to satisfy the \emph{$\de\de^\Lambda$-Lemma} if every $\de$-exact, $\de^\Lambda$-closed form is $\de\de^\Lambda$-exact, namely, if $H_{\de+\de^\Lambda}^{\bullet}(X;\R)\to H_{dR}^{\bullet}(X;\R)$ is injective. Furthermore, one says that the \emph{Hard Lefschetz Condition} holds on $X$ if
\begin{equation}\label{eq:hlc}
\tag{HLC}
 \text{for every } k\in\N\;, \qquad L^k\colon H^{n-k}_{dR}(X;\R) \stackrel{\simeq}{\longrightarrow} H^{n+k}_{dR}(X;\R) \;.
\end{equation}

\medskip

In fact, by \cite[Corollary 2]{mathieu}, \cite[Theorem 0.1]{yan}, \cite[Proposition 1.4]{merkulov}, \cite{guillemin}, \cite[Theorem 5.4]{cavalcanti}, (compare also \cite{deligne-griffiths-morgan-sullivan}), it turns out that the following conditions are equivalent:
 \begin{itemize}
  \item $X$ satisfies the $\de\de^\Lambda$-Lemma;
  \item the natural homomorphism $H_{\de+\de^\Lambda}^{\bullet}(X;\R)\to H_{dR}^{\bullet}(X;\R)$ is actually an isomorphism;
  \item every de Rham cohomology class admits a representative being both $\de$-closed and $\de^\Lambda$-closed (i.e., Brylinski's conjecture \cite[Conjecture 2.2.7]{brylinski} holds on $X$);
  \item the Hard Lefschetz Condition holds on $X$.
 \end{itemize}

\medskip

\subsection{Primitive currents}
Denote by $\correnti_\bullet X:=:\correnti^{2n-\bullet}X$ the space of currents (that is, the topological dual space of $\wedge^\bullet X$, endowed with the usual topology, see, e.g., \cite[\S9]{derham}). The differential $\de\colon \correnti_\bullet X \to \correnti_{\bullet-1}X$ is defined by duality from $\de\colon \wedge^{\bullet-1}X \to \wedge^{\bullet}X$. Hence, one can consider the \emph{de Rham homology} $H_\bullet^{dR}(X;\R)$ as the homology of the complex $\left(\correnti_\bullet X,\,\de\right)$. One has a natural injective homomorphism given by
$$
\begin{array}{rclcl}
 T_\sspace &\colon& \wedge^\bullet X &\to    & \correnti^\bullet X \;,\\[5pt]
           &      & \alpha           &\mapsto& \int_X \alpha\wedge\sspace \;;
\end{array}
$$
since $\de T_{\sspace}=T_{\de\sspace}$, the homomorphism $T_\sspace$ induces a map
$$ T_\sspace\colon H^\bullet_{dR}(X;\R) \to H_\bullet^{dR}(X;\R) \;. $$
Moreover, see \cite[Theorem 14]{derham}, one has an isomorphism
$$ H^\bullet_{dR}(X;\R) \;\simeq\; H_\bullet^{dR}(X;\R) \;;$$
in particular, $T_\sspace \colon H^\bullet_{dR}(X;\R)\to H_\bullet^{dR}(X;\R)$ is an isomorphism.

\medskip

Following \cite[Definition 5.1]{lin}, set, by duality,
\begin{eqnarray*}
L\colon \correnti_\bullet X\to \correnti_{\bullet-2}X\;, \quad && S \mapsto S\left(L\, \sspace\right) \;,\\[5pt]
\Lambda\colon \correnti_\bullet X\to \correnti_{\bullet+2}X\;, \quad && S \mapsto S\left(\Lambda\, \sspace\right) \;,\\[5pt]
H\colon \correnti_\bullet X\to \correnti_\bullet X\;, \quad && S \mapsto S\left(-H\, \sspace\right) \;,
\end{eqnarray*}

A current $S\in \correnti^k X$ is said \emph{primitive} if $\Lambda S=0$, equivalently, if $L^{n-k+1}S=0$ (see, e.g., \cite[Proposition 5.3]{lin}); denote by $\Primcorrenti^\bullet X:=:\Primcorrenti_{2n-\bullet} X$ the space of primitive currents on $X$.

The following results are proven by Y. Lin in \cite{lin} and provide a Lefschetz decomposition also on the space of currents, respectively on the space of flat currents.

\begin{prop}[{\cite[Lemma 5.2, Lemma 5.12]{lin}}]
 Let $X$ be a closed manifold endowed with a symplectic structure. Then
 \begin{itemize}
  \item $\left\langle L,\, \Lambda,\, H\right\rangle$ gives a $\mathfrak{sl}(2;\R)$-module structure on $\correnti^\bullet X$;
  \item $\left\langle L,\, \Lambda,\, H\right\rangle$ gives a $\mathfrak{sl}(2;\R)$-module structure on the space of flat currents.
 \end{itemize}
\end{prop}

\noindent In particular, we get a Lefschetz decomposition on the space of currents, \cite[Proposition 5.3]{lin}:
$$ \correnti^\bullet X \;=\; \bigoplus_{r\in\N} L^r\, \Primcorrenti^{\bullet-2r}X \;:=\; \bigoplus_{r\in\N} L^r\, \Primcorrenti_{2n-\bullet+2r}X \;.$$
Finally, if $j\colon Y\hookrightarrow X$ is a compact submanifold of $X$ of codimension $k$, then it is defined the {\em dual current} $\rho_Y\in\correnti_k X$ associated with $Y$, by setting
$$
 \rho_Y(\varphi) \;:=\; \int_Y j^*(\varphi) \:,
$$
for every test form $\varphi\in\wedge^kX$. If $Y$ is a closed submanifold, then the dual current $\rho_Y$ is closed, and, according to \cite[Lemma 4.1]{tseng-yau-1}, $\rho_Y$ is primitive if and only if $Y$ is co-isotropic.

\section{Symplectic (co)homology decomposition}\label{sec:cohomology}
In this section, we provide a symplectic counterpart to T.-J. Li and W. Zhang's theory on cohomology of almost-complex manifolds developed in \cite{li-zhang}.

\medskip

Let $X$ be a $2n$-dimensional closed manifold endowed with a symplectic structure $\omega$. For any $r,s\in\N$, define
$$ H^{(r,s)}_\omega(X;\R) \;:=\; \left\{\left[L^r\,\beta^{(s)}\right]\in H^{2r+s}_{dR}(X;\R) \st \beta^{(s)}\in \Prim^sX \right\} \;\subseteq\; H^{2r+s}_{dR}(X;\R) \;.$$
Obviously, for every $k\in\N$, one has
$$ \sum_{2r+s=k} H^{(r,s)}_\omega(X;\R) \;\subseteq\; H^k_{dR}(X;\R) \;:$$
we are concerned in studying when the above inclusion is actually an equality, and when the sum is actually a direct sum.

\begin{rem}
We underline the relations between the above subgroups and the primitive cohomologies introduced by L.-S. Tseng and S.-T. Yau in \cite{tseng-yau-1}.\\
As regards L.-S. Tseng and S.-T. Yau's primitive $\left(\de+\de^\Lambda\right)$-cohomology $\PH^{\bullet}_{\de+\de^{\Lambda}}(X;\R)$, note that, for every $r,s\in\N$,
$$ \imm\left(L^r\,\PH^{s}_{\de+\de^{\Lambda}}(X;\R)\to H^\bullet_{dR}(X;\R)\right) \;=\; L^r\, H^{(0,s)}_\omega(X;\R) \;\subseteq\; H^{(r,s)}_\omega(X;\R) \;. $$
In \cite[\S4.1]{tseng-yau-1}, L.-S. Tseng and S.-T. Yau have introduced also the primitive cohomology groups
$$ \PH^s_{\de}(X;\R) \;:=\; \frac{\ker\de\cap\ker\de^\Lambda\cap\Prim^sX}{\imm\de\lfloor_{\Prim^{s-1}X\cap\ker\de^\Lambda}} \;, $$
where $s\in\N$, proving that the homology on co-isotropic chains is naturally dual to $\PH^{2n-\bullet}_{\de}(X;\R)$, see \cite[page 41]{tseng-yau-1};
in \cite[Lemma 2.7]{lin}, Y. Lin has proved that, if the Hard Lefschetz Condition holds on $X$, then
$$ H^{(0,\bullet)}_\omega(X;\R) \;=\; \PH^\bullet_{\de}(X;\R) \;. $$
\end{rem}

\begin{rem}
 In \cite{conti-tomassini}, D. Conti and the second author studied the notion of half-flat structure on a $6$-dimensional manifold $X$ (see \cite{chiossi-salamon}). Namely, an \emph{$\mathrm{SU}(3)$-structure} $\left(\omega,\, \psi\right)$ on $X$ (where $\omega$ is a non-degenerate real $2$-form, and $\psi$ is a decomposable complex $3$-form, such that $\psi \wedge \omega = 0$ and $\psi\wedge\bar\psi = -\frac{4\im}{3}\,\omega^ 3$) is called \emph{half-flat} if both $\omega\wedge\omega$ and $\Re\mathfrak{e}\, \psi$ are $\de$-closed. Note in particular that, if $\left(\omega,\, \psi\right)$ is a symplectic half-flat structure on $X$, then $\left[\Re\mathfrak{e}\,\psi\right]\in H^{(0,3)}_\omega(X;\R)$. Furthermore, $\Re\mathfrak{e}\,\psi$ is a calibration on $X$ and special Lagrangian submanifolds are naturally defined also in this context.
\end{rem}

\begin{rem}
A class of example of closed symplectic manifolds satisfying the cohomology decomposition by means of the above subgroups $H^{\bullet,\bullet}_\omega(X;\R)$ (actually, satisfying an even stronger cohomology decomposition) is provided by the closed symplectic manifolds satisfying the $\de\de^\Lambda$-Lemma, equivalently, by the Hard Lefschetz Condition. More precisely, recall that, by \cite{yan} (see also \cite[Theorem 3.11, Proposition 3.13]{tseng-yau-1}), for a $2n$-dimensional closed manifold $X$ endowed with a symplectic structure $\omega$, the following conditions are equivalent:
 \begin{itemize}
  \item $X$ satisfies the $\de\de^\Lambda$-Lemma;
  \item it holds that
    $$ H^\bullet_{dR}(X;\R) \;=\; \bigoplus_{r\in\N} L^r\, H^{(0,\bullet-2r)}_\omega(X;\R) \;.$$
 \end{itemize}
\end{rem}

\medskip

Analogously, considering the space $\correnti^\bullet X$ of currents and the de Rham homology $H_\bullet^{dR}(X;\R)$, for every $r,s\in\N$, define
$$ H_{(r,s)}^\omega(X;\R) \;:=\; \left\{\left[L^r\, B_{(s)}\right]\in H_{-2r+s}^{dR}(X;\R) \st B_{(s)}\in \Primcorrenti_sX \right\} \;\subseteq\; H_{-2r+s}^{dR}(X;\R) \;;$$
as previously, for every $k\in\N$, we have just the inclusion
$$ \sum_{-2r+s=k} H_{(r,s)}^\omega(X;\R) \;\subseteq\; H_k^{dR}(X;\R) \;,$$
but, in general, neither the sum is direct nor the inclusion is an equality.

\medskip

We prove that, fixed $k\in\N$, if the sum $\sum_{2r+s=2n-k}H^{(r,s)}_\omega(X;\R)$ gives the whole \kth{\left(2n-k\right)} de Rham cohomology group, then the sum of the subgroups of the \kth{k} de Rham cohomology group is direct (this result should be compared with \cite[Proposition 2.30]{li-zhang}, see also \cite[Theorem 2.1]{angella-tomassini}, in the almost-complex case).

\begin{prop}\label{prop:full-k-pure-2n-k}
Let $X$ be a $2n$-dimensional closed manifold endowed with a symplectic structure $\omega$.
For every $k\in\N$, the following implications hold:
$$
\xymatrix{
 H^{k}_{dR}(X;\R) = \sum_{2r+s=k}H^{(r,s)}_\omega(X;\R) \ar@{=>}[r] \ar@{=>}[d] & \bigoplus_{-2r+s=k} H_{(r,s)}^\omega(X;\R) \;\subseteq\; H_k^{dR}(X;\R) \ar@{=>}[d] \\
 H_{2n-k}^{dR}(X;\R) = \sum_{-2r+s=2n-k}H_{(r,s)}^\omega(X;\R) \ar@{=>}[r] & \bigoplus_{2r+s=2n-k} H^{(r,s)}_\omega(X;\R) \;\subseteq\; H^{2n-k}_{dR}(X;\R) \;. \\
}
$$
\end{prop}

\begin{proof}
 Note that the quasi-isomorphism $T_\sspace\colon \wedge^\bullet X\to\correnti^\bullet X$ satisfies
 $$ T_{L\, \sspace} \;=\; L T_\sspace \;, $$
 and hence, in particular, it preserves the bi-graduation,
 $$ T \left(L^{\bullet}\, \Prim^{\bullet}X\right) \;\subseteq\; L^{\bullet}\, \Primcorrenti^{\bullet}X \;:=:\; L^{\bullet}\, \Primcorrenti_{2n-\bullet}X \;, $$
 and it induces, for every $r,s\in\N$, an injective map
 $$ H^{(r,s)}_\omega(X;\R) \hookrightarrow H_{(r,2n-s)}^\omega(X;\R) \;. $$
 Therefore the two vertical implications are proven.\\
 Consider now the non-degenerate duality pairing
 $$ \left\langle \sspace,\,\ssspace\right\rangle \colon H^{\bullet}_{dR}(X;\R) \times H_{\bullet}^{dR}(X;\R) \to \R\;, $$
 and note that, for every $r,s\in\N$,
 $$ \ker\left\langle H^{(r,s)}_{\omega}(X;\R),\;\sspace\right\rangle \;\supseteq\; \sum_{(p,q)\neq(n-r-s,2n-s)} H_{(p,q)}^{\omega}(X;\R) \;.$$
 Arguing in the same way in the case of currents, this suffices to prove the two horizontal implications.
\end{proof}

A straightforward consequence of \cite{yan} and Proposition \ref{prop:full-k-pure-2n-k} is the following result, which should be compared with \cite[Theorem 2.16, Proposition 2.17]{draghici-li-zhang}.

\begin{cor}\label{cor:hlc}
 Let $X$ be a closed manifold endowed with a symplectic structure $\omega$. Suppose that the Hard Lefschetz Condition holds on $X$, equivalently, that $X$ satisfies the $\de\de^\Lambda$-Lemma. Then
 $$ H^\bullet_{dR}(X;\R) \;=\; \bigoplus_{r\in\N} H^{(r,\bullet-2r)}_\omega(X;\R) $$
 and
 $$ H_\bullet^{dR}(X;\R) \;=\; \bigoplus_{r\in\N} H_{(r,\bullet+2r)}^\omega(X;\R) \;.$$
\end{cor}

In particular, when $\dim X=4$ and taking $k=2$ in Proposition \ref{prop:full-k-pure-2n-k}, one gets that, if $H^{2}_{dR}(X;\R)=H^{(1,0)}_\omega(X;\R) + H^{(0,2)}_\omega(X;\R)$ holds, then actually $H^2_{dR}(X;\R) = H^{(1,0)}_\omega(X;\R) \oplus H^{(0,2)}_\omega(X;\R)$ holds. In fact, the following result states that $H^2_{dR}(X;\R)$ always decomposes as direct sum of $H^{(1,0)}_\omega(X;\R)$ and $H^{(0,2)}_\omega(X;\R)$, also in dimension higher than $4$: this gives a symplectic counterpart to T. Dr\v{a}ghici, T.-J. Li and W. Zhang's \cite[Theorem 2.3]{draghici-li-zhang} in the complex setting (in fact, without the restriction to dimension $4$).

\begin{thm}\label{thm:sympl-decomp-H2}
 Let $X$ be a closed manifold endowed with a symplectic structure $\omega$. Then
 \begin{eqnarray*}
  H^2_{dR}(X;\R) &=& H^{(1,0)}_\omega(X;\R) \oplus H^{(0,2)}_\omega(X;\R) \;.
 \end{eqnarray*}
 In particular, if $\dim X=4$, then
 $$ H^{\bullet}_{dR}(X;\R) \;=\; \bigoplus_{r\in\N} H^{(r,\bullet-2r)}_\omega(X;\R) $$
 and
 $$ H_{\bullet}^{dR}(X;\R) \;=\; \bigoplus_{r\in\N} H_{(r,\bullet+2r)}^\omega(X;\R) \;.$$
\end{thm}

\begin{proof}
 Firstly, we prove that $H^{(1,0)}_\omega(X;\R) \cap H^{(0,2)}_{\omega}(X;\R)=\{0\}$. Let
 $$
 \mathfrak{c}:=:\left[f\,\omega\right]:=:\left[\beta^{(2)}\right]\in H^{(1,0)}_\omega(X;\R) \cap H^{(0,2)}_{\omega}(X;\R)\,,
 $$ where $f\in\mathcal{C}^\infty(X;\R)$ and $\beta^{(2)}\in \Prim^{2}X$. Being $\Prim^{2}X=\ker L^{n-1}\lfloor_{\wedge^{2}X}$, one has
 \begin{eqnarray*}
  0 &=& \int_X f\,L^{n-1} \beta^{(2)} \;=\; \int_X f\,\omega \wedge \beta^{(2)} \wedge \omega^{n-2} \\[5pt]
    &=& \int_X f\,\omega \wedge f\,\omega \wedge \omega^{n-2} \;=\; \int_X f^2 \omega^n
 \end{eqnarray*}
 hence $f=0$, that is, $\mathfrak{c}=0$.\\
 Now, we prove that $H^{2}_{dR}(X;\R)=H^{(1,0)}_{\omega}(X;\R)+H^{(0,2)}_{\omega}(X;\R)$. Suppose that $$
 \mathfrak{a}:=:\left[\alpha\right]\not\in H^{(1,0)}_\omega(X;\R) + H^{(0,2)}_\omega(X;\R)\,,
  $$
  that is, for every $\gamma\in\wedge^1X$, one has $\alpha+\de\gamma \not\in L\wedge^0X$ and $\alpha+\de\gamma \not\in \Prim^2X$. In particular, asking that $\alpha+\de\gamma \not\in \Prim^2X=\ker L^{n-1}\lfloor_{\wedge^2X}$ means that $\left[L^{n-1}\alpha\right]\neq 0$ in $H^{2n}_{dR}(X;\R)$, since $L^{n-1}\colon\wedge^1X\to\wedge^{2n-1}X$ is an isomorphism. Since $L\colon\wedge^1X\to \wedge^3X$ is injective, asking that $\alpha+\de\gamma \not\in L\wedge^0X$, we get that $\left[\alpha\right] \neq \lambda\,\left[\omega\right]$ for any $\lambda\in\R\setminus\{0\}$. Hence we get that $\left[L^{n-1}\alpha\right] \not\in \left\langle \left[\omega^n\right] \right\rangle$, which is absurd.
\end{proof}

\begin{rem}
 Note that the argument in the proof of Theorem \ref{thm:sympl-decomp-H2} can be generalized to prove the following:\\
{\itshape If $X$ be a $2n$-dimensional closed manifold endowed with a symplectic structure $\omega$, then, for every $k\in\left\{1,\ldots,\left\lfloor\frac{n}{2}\right\rfloor\right\}$, it holds
$$ H^{(k,0)}_\omega(X;\R) \cap H^{(0,2k)}_{\omega}(X;\R) \;=\; \left\{0\right\} \;.$$}
\end{rem}

\medskip

In some cases, the study of the spaces $H^{(r,s)}_\omega(X;\R)$ can be reduced to the study of $H^{(0,r)}_\omega(X;\R)$: this is the matter of the following result.

\begin{prop}\label{prop:livello<k}
 Let $X$ be a $2n$-dimensional closed manifold endowed with a symplectic structure $\omega$. Then, for every $r,s\in\N$ such that $2r+s\leq n$, one has
$$ H^{(r,s)}_\omega(X;\R) \;=\; L^r H^{(0,s)}_\omega(X;\R) \;.$$
\end{prop}

\begin{proof}
 Since $L\colon\wedge^{j}X\to\wedge^{j+2}X$ is injective for $j\leq n-1$ (in fact, an isomorphism for $j=n-1$), and $\left[\de,\,L\right]=0$, we get that
 \begin{eqnarray*}
  H^{(r,s)}_{\omega}(X;\R) &=& \left\{ \left[\omega^r\,\beta^{(s)}\right]\in H^{2r+s}_{dR}(X;\R) \st \beta^{(s)}\in\wedge^{s}X\cap\ker\Lambda \text{ such that }L^r\de \beta^{(s)}=0\right\} \\[5pt]
  &=& \left\{ \left[\omega^r\right]\smile\left[\beta^{(s)}\right]\in H^{2r+2}_{dR}(X;\R) \st \beta^{(s)}\in\wedge^sX\cap\ker\Lambda \right\} \,,
 \end{eqnarray*}
 assumed that $2r+s\leq n$.
\end{proof}

\noindent In particular, for every $r\in\left\{1,\ldots,\left\lfloor\frac{n}{2}\right\rfloor\right\}$, the spaces $H^{(r,0)}_\omega(X;\R)$ are $1$-dimensional $\R$-vector spaces: more precisely, $H^{(r,0)}_{\omega}(X;\R)=\R\left\langle \left[\omega^r\right] \right\rangle$.\\
In particular, by the previous result follows that, for $k\leq n$, the condition 
$$
H^k_{dR}(X;\R) = \bigoplus_{r\in\N}H^{(r,k-2r)}_\omega(X;\R)
$$ 
is in fact equivalent to $H^k_{dR}(X;\R) = \bigoplus_{r\in\N} L^r\, H^{(0,k-2r)}_\omega(X;\R)$.

\section{Symplectic (co)homology decomposition on solvmanifolds}\label{sec:solvmanifolds}
By a \emph{nilmanifold} (respectively, a \emph{solvmanifold}) we mean a compact quotient of a nilpotent (respectively, solvable) Lie group by a discrete co-compact subgroup. A solvmanifold $X=\left.\Gamma\right\backslash G$ is called \emph{completely-solvable} if, for any $g\in G$, all the eigenvalues of $\mathrm{Ad} g\in\mathrm{End}(G)$ are real, equivalently, if, for any $X\in\mathfrak{g}$, all the eigenvalues of $\mathrm{ad} X\in\mathrm{End}(\mathfrak{g})$ are real.\\
To shorten the notation, we will refer to a given solvmanifold $X=\left.\Gamma\right\backslash G$ writing the structure equations of its Lie algebra: for example, writing
$$ X \;:=\; \left(0^4,\; 12,\; 13\right)\;, $$
we mean that there exists a basis of the naturally associated Lie algebra $\mathfrak{g}$, let us say $\left\{e_1,\,\ldots,\,e_6\right\}$, whose dual will be denoted by $\left\{e^1,\,\ldots,\,e^6\right\}$ and with respect to which the structure equations are
$$
\left\{
\begin{array}{l}
 \de e^1 \;=\; \de e^2 \;=\; \de e^3 \;=\; \de e^4 \;=\; 0 \\[5pt]
 \de e^5 \;=\; e^{1}\wedge e^{2} \;=:\; e^{12} \\[5pt]
 \de e^6 \;=\; e^{1}\wedge e^{3} \;=:\; e^{13}
\end{array}
\right. \;,
$$
where we also shorten $e^{AB}:=e^A\wedge e^B$.
Recall that, by A.~I. Mal'tsev's theorem \cite[Theorem 7]{malcev}, given a nilpotent Lie algebra $\mathfrak{g}$ with rational structure constants, then the connected simply-connected Lie group $G$ naturally associated to $\mathfrak{g}$ admits a co-compact discrete subgroup $\Gamma$, and hence there exists a nilmanifold $X:=\left.\Gamma\right\backslash G$ whose Lie algebra is $\mathfrak{g}$.
Dealing with \emph{$G$-left-invariant} objects on $X$, we mean objects induced by objects on $G$ which are invariant under the left-action of $G$ on itself given by left-translations. By means of left-translations, $G$-left-invariant objects will be identified with objects on the Lie algebra.\\
By A. Hattori's theorem \cite[Corollary 4.2]{hattori}, the cohomology of a completely-solvable solvmanifold $X$ is isomorphic to the cohomology $H^\bullet\left(\mathfrak{g};\R\right):=H^\bullet\left(\wedge^\bullet\duale{\mathfrak{g}},\,\de\right)$ of the complex $\left(\wedge^\bullet\duale{\mathfrak{g}},\,\de\right)$, where $\de\colon \wedge^{\bullet}\duale{\mathfrak{g}} \to \wedge^{\bullet+1}\duale{\mathfrak{g}}$ is induced by $\de_\mathfrak{g}\colon\wedge^1\duale{\mathfrak{g}}\to\wedge^2\duale{\mathfrak{g}}$, $\left(\de_\mathfrak{g}\alpha\right)(x,y):=-\alpha\left(\left[x,y\right]\right)$: for simplicity, in writing the cohomology of a solvmanifolds, we list the harmonic representatives with respect to the $G$-left-invariant metric $g:=\sum_\ell e^\ell\odot e^\ell$ instead of their classes.\\
We recall that, by Ch. Benson and C.~S. Gordon's theorem \cite[Theorem A]{benson-gordon}, if a nilmanifold $X$ is endowed with a symplectic structure $\omega$ such that the Hard Lefschetz Condition holds, then it is diffeomorphic to a torus.

\medskip

Let $X=\left.\Gamma\right\backslash G$ be a completely-solvable solvmanifold, endowed with a $G$-left-invariant structure $\omega$. In particular, $\omega$ being $G$-left-invariant, $\left\langle L,\, \Lambda,\, H\right\rangle$ induces a $\mathfrak{sl}(2;\R)$-representation both on $\wedge^\bullet X$ and on its (quasi-isomorphic) subspace made of the $G$-left-invariant forms (which is isomorphic to $\wedge^\bullet\duale{\mathfrak{g}}$). For any $r,s\in\N$, we can consider both the subgroup $H^{(r,s)}_\omega(X;\R)$ of $H^\bullet_{dR}(X;\R)$, and the subgroup
$$ H^{(r,s)}_\omega(\mathfrak{g};\R) \;:=\; \left\{ \left[L^r\,\beta^{(s)}\right]\in H^{\bullet}\left(\mathfrak{g};\R\right) \st \Lambda \beta^{(s)}=0 \right\} $$
of $H^{\bullet}\left(\mathfrak{g};\R\right) \simeq H^\bullet_{dR}(X;\R)$, namely, the subgroup made of the de Rham cohomology classes admitting $G$-left-invariant representatives in $L^r \, \Prim^sX$.\\
In this section, we are concerned in studying the linking between $H^{(\bullet,\bullet)}_\omega(X;\R)$ and $H^{(\bullet,\bullet)}_\omega(\mathfrak{g};\R)$. This will let us study explicit examples.

\medskip

First of all, we will need the following lemma by J. Milnor.

\begin{lemma}[{\cite[Lemma 6.2]{milnor}}]\label{lemma:milnor}
 Any connected Lie group that admits a discrete subgroup with compact quotient is unimodular and in particular admits a bi-invariant volume form $\eta$.
\end{lemma}

In the following lemma, we recall \emph{F.~A. Belgun's symmetrization trick}, see \cite[Theorem 7]{belgun} and \cite[Theorem 2.1]{fino-grantcharov}.

\begin{lemma}\label{lemma:fino-grantcharov}
Let $X=\left.\Gamma\right\backslash G$ be a solvmanifold and call $\mathfrak{g}$ the Lie algebra naturally associated to the connected simply-connected Lie group $G$. Let $\omega$ be a $G$-left-invariant symplectic structure on $X$.
Let $\eta$ be the $G$-bi-invariant volume form on $G$ given by J. Milnor's Lemma \ref{lemma:milnor} and such that $\int_X\eta=1$. (Up to identifying $G$-left-invariant forms with linear forms over $\duale{\mathfrak{g}}$ through left-translations,) define the map
$$ \mu\colon \wedge^\bullet X \to \wedge^\bullet \duale{\mathfrak{g}}\;,\qquad \mu(\alpha)\;:=\;\int_X \alpha\lfloor_m \, \eta(m) \;.$$
One has that
$$ \mu\lfloor_{\wedge^\bullet \duale{\mathfrak{g}}}\;=\;\id\lfloor_{\wedge^\bullet \duale{\mathfrak{g}}} $$
and that
$$  \sharp\left(\mu(\sspace)\right) \;=\; \mu\left(\sharp\sspace\right) \;, \qquad \text{ for } \qquad \sharp\in\left\{\de,\, L\right\} \;.$$
In particular, $\mu$ sends primitive forms to $G$-left-invariant primitive forms.
\end{lemma}

\begin{proof}
 It has to be shown just that $\mu\left(L\,\alpha\right)=L\,\mu\left(\alpha\right)$ for every $\alpha\in\wedge^\bullet X$. Note that, being $\omega$ a $G$-left-invariant form, one has $\mu\left(L\,\alpha\right)=\int_X \left(\omega\wedge\alpha\right)\lfloor_m \, \eta(m) =\int_X \omega\lfloor_m \wedge \alpha\lfloor_m \, \eta(m) = \omega \wedge \int_X \alpha\lfloor_m \, \eta(m) = L\,\mu\left(\alpha\right)$.
\end{proof}

Then we can prove the following result, which relates the subgroups $H^{(r,s)}_\omega(X;\R)$ with their invariant part $H^{(r,s)}_\omega(\mathfrak{g};\R)$ (compare it with \cite[Proposition 2.4]{angella-rossi} for almost-$\mathbf{D}$-complex structures in the sense of F.~R. Harvey and H.~B. Lawson, and also with \cite[Theorem 3.4]{fino-tomassini} for almost-complex structures).

\begin{prop}\label{prop:linear-cpf-invariant-cpf}
Let $X=\left.\Gamma\right\backslash G$ be a solvmanifold endowed with a $G$-left-invariant symplectic structure $\omega$. Call $\mathfrak{g}$ the Lie algebra naturally associated to the connected simply-connected Lie group $G$.
For every $r,s\in\N$, the map
$$ j\colon H^{(r,s)}_\omega(\mathfrak{g};\R) \to H^{(r,s)}_\omega(X;\R) $$
induced by left-translations is injective, and, if $H^\bullet\left(\mathfrak{g};\R\right) \simeq H^\bullet_{dR}(X;\R)$ (for instance, if $X$ is a completely-solvable solvmanifold), then it is in fact an isomorphism.
\end{prop}

\begin{proof}
Left-translations induce the map $j\colon H^{(r,s)}_\omega(\mathfrak{g};\R) \to H^{(r,s)}_\omega(X;\R)$.
Consider the F.~A. Belgun's symmetrization map $\mu\colon\wedge^\bullet X \to \wedge^\bullet \duale{\mathfrak{g}}$: by Lemma \ref{lemma:fino-grantcharov}, since it commutes with $\de$, it induces the map $\mu\colon H^\bullet_{dR}(X;\R) \to H^\bullet\left(\mathfrak{g};\R\right)$, and, since it commutes with $L$ and $\Lambda$, it induces the map $\mu\colon H^{(r,s)}_\omega(X;\R) \to H^{(r,s)}_\omega(\mathfrak{g};\R)$. Moreover, since $\mu$ is the identity on the space of $G$-left-invariant forms, we get the commutative diagram
$$
\xymatrix{
H^{(r,s)}_{\omega}(\mathfrak{g};\R) \ar[r]^{j} \ar@/_1.5pc/[rr]_{\id} & H^{(r,s)}_{\omega}(X;\R) \ar[r]^{\mu} & H^{(r,s)}_{\omega}(\mathfrak{g};\R)
}
$$
hence $j\colon H^{(r,s)}_\omega(\mathfrak{g};\R) \to H^{(r,s)}_\omega(X;\R)$ is injective, and $\mu\colon H^{(r,s)}_\omega(X;\R) \to H^{(r,s)}_\omega(\mathfrak{g};\R)$ is surjective.

Furthermore, when $H^\bullet\left(\mathfrak{g};\R\right) \simeq H^\bullet_{dR}(X;\R)$ (for instance, when $X$ is a completely-solvable solvmanifold, by Hattori's theorem \cite[Theorem 4.2]{hattori}), since $\mu\lfloor_{\wedge^\bullet\duale{\mathfrak{g}}} = \id\lfloor_{\wedge^\bullet\duale{\mathfrak{g}}}$, we get that $\mu\colon H^\bullet_{dR}(X;\R) \to H^\bullet\left(\mathfrak{g};\R\right)$ is the identity map, and hence $\mu\colon H^{(r,s)}_\omega(X;\R) \to H^{(r,s)}_\omega(\mathfrak{g};\R)$ is also injective, hence an isomorphism.
\end{proof}

\medskip

Proposition \ref{prop:linear-cpf-invariant-cpf} is a useful tool to study explicit examples.

\begin{ex}
Take the $6$-dimensional nilmanifold
$$ X \;:=\; \left(0^3,\; 12,\; 14-23,\; 15+34 \right) $$
endowed with the left-invariant symplectic structure
$$ \omega \;:=\; e^{16}+e^{35}+e^{24} \;.$$

By Nomizu's theorem \cite[Theorem 1]{nomizu}, one computes
\begin{eqnarray*}
H^1_{dR}(X;\R) &=& \underbrace{\R\left\langle e^1,\; e^2,\; e^3\right\rangle}_{=H^{(0,1)}_\omega(X;\R)} \;,\\[5pt]
H^2_{dR}(X;\R) &=& \underbrace{\R\left\langle e^{16}+e^{35}+e^{24} \right\rangle}_{=H^{(1,0)}_\omega(X;\R)} \oplus \underbrace{\R\left\langle e^{13},\; e^{14}+e^{23},\; 2\cdot e^{24}-e^{16}-e^{35}\right\rangle}_{=H^{(0,2)}_\omega(X;\R)} \;, \\[5pt]
H^3_{dR}(X;\R) &=& \R\left\langle e^{126}-e^{145}-2\cdot e^{235},\; e^{136},\; e^{146}+\frac{1}{2}\cdot e^{236}+\frac{1}{2}\cdot e^{345},\; e^{245}\right\rangle \;. \\[5pt]
\end{eqnarray*}

Since the Lefschetz decompositions of the $g$-harmonic representatives of $H^3_{dR}(X;\R)$ are
\begin{eqnarray*}
 e^{126}-e^{145}-2\cdot e^{235} &=& \underbrace{\left(-\frac{1}{2}\cdot e^{126}-\frac{1}{2}\cdot e^{235}-e^{145}\right)}_{\in \Prim^3X} + \underbrace{\left(\frac{3}{2}\cdot e^{126}-\frac{3}{2}\cdot e^{235}\right)}_{=\; L\left(-\frac{3}{2}\cdot e^2\right)} \;, \\[5pt]
 e^{136} &=& \underbrace{\left(\frac{1}{2}\cdot e^{136}-\frac{1}{2}\cdot e^{234}\right)}_{\in \Prim^3X} + \underbrace{\left(\frac{1}{2}\cdot e^{136}+\frac{1}{2}\cdot e^{234}\right)}_{=\; L\left(-\frac{1}{2}\cdot e^3\right)} \;, \\[5pt]
 e^{146}+\frac{1}{2}\cdot e^{236}+\frac{1}{2}\cdot e^{345} &=& \underbrace{\left(\frac{1}{4}\cdot e^{146}-\frac{1}{4}\cdot e^{345}+\frac{1}{2}\cdot e^{236}\right)}_{\in \Prim^3X} + \underbrace{\left(\frac{3}{4}\cdot e^{146}+\frac{3}{4}\cdot e^{345}\right)}_{=\; L\left(-\frac{3}{4}\cdot e^4\right)} \;, \\[5pt]
 e^{245} &=& \underbrace{\left(\frac{1}{2}\cdot e^{156}+\frac{1}{2}\cdot e^{245}\right)}_{\in \Prim^3X} + \underbrace{\left(-\frac{1}{2}\cdot e^{156}+\frac{1}{2}\cdot e^{245}\right)}_{=\; L\left(\frac{1}{2}\cdot e^5\right)} \;,
\end{eqnarray*}
and since
$$ \de\wedge^2\duale{\mathfrak{g}} \;=\; \R\left\langle e^{123},\, e^{124},\; e^{125},\; e^{126}+e^{145},\; e^{134},\; e^{135},\; e^{146}-e^{236}-e^{345},\; e^{234} \right\rangle \;, $$
we get that
\begin{eqnarray*}
 \left[e^{126}-e^{145}-2\cdot e^{235}\right] &=& \left[e^{126}-e^{145}-2\cdot e^{235}+\de e^{46}\right] \\[5pt]
 &=& \left[2\cdot e^{126}-2\cdot e^{235}\right] \;=\; \left[L \left(-2\cdot e^2\right)\right] \;\in\; H^{(1,1)}_\omega(X;\R)
\end{eqnarray*}
and
\begin{eqnarray*}
 \left[e^{136}\right] &=& \left[e^{136}+\de\left(\frac{1}{2}\cdot e^{45}-\frac{1}{2}\cdot e^{26}\right)\right] \;=\; \left[e^{136} + e^{234}\right] \;=\; \left[L\left(-e^3\right)\right] \;\in\; H^{(1,1)}_\omega(X;\R) \;, \\[5pt]
 \left[e^{136}\right] &=& \left[e^{136}-\de\left(\frac{1}{2}\cdot e^{45}-\frac{1}{2}\cdot e^{26}\right)\right] \;=\; \left[e^{136} - e^{234}\right] \;\in\; H^{(0,3)}_\omega(X;\R) \;,
\end{eqnarray*}
while it is straightforward to check that
$$ \R\left\langle e^{146}+\frac{1}{2}\cdot e^{236}+\frac{1}{2}\cdot e^{345},\; e^{245}\right\rangle \cap \left(H^{(0,3)}_\omega(X;\R)+H^{(1,1)}_\omega(X;\R)\right) \;=\; \left\{0\right\} \;; $$
in particular, $H^{(0,3)}_\omega(X;\R)+H^{(1,1)}_\omega(X;\R)\subsetneq H^3_{dR}(X;\R)$ and $H^{(0,3)}_\omega(X;\R)\cap H^{(1,1)}_\omega(X;\R)\neq\{0\}$.
\end{ex}

\begin{ex}
Take the $6$-dimensional solvable Lie algebra
$$ \mathfrak{g}_{3.4}^{-1}\oplus\mathfrak{g}_{3.5}^0 \;:=\; \left(-13,\; 23,\; 0,\; -56,\; 46,\; 0 \right) $$
endowed with the linear symplectic structure
$$ \omega \;:=\; e^{12}+e^{36}+e^{45} \;.$$
The corresponding connected simply-connected Lie group admits a compact quotient, whose de Rham cohomology is the same as the cohomology of $\left(\wedge^\bullet\duale{\left(\mathfrak{g}_{3.4}^{-1}\oplus\mathfrak{g}_{3.5}^0\right)},\, \de\right)$, see \cite[Table 5]{bock}.

It is straightforward to compute
\begin{eqnarray*}
H_{dR}^1\left(X;\R\right) &=& \underbrace{\R\left\langle e^3,\; e^6 \right\rangle}_{=\; H^{(0,1)}_\omega\left(X;\R\right)} \;, \\[5pt]
H_{dR}^2\left(X;\R\right) &=& \underbrace{\R\left\langle e^{12}+e^{36}+e^{45} \right\rangle}_{=\; H^{(1,0)}_\omega\left(X;\R\right)} \oplus \underbrace{\R\left\langle e^{12}-e^{36},\; e^{12}-e^{45} \right\rangle}_{=\; H^{(0,2)}_\omega\left(X;\R\right)} \;, \\[5pt]
H_{dR}^3\left(X;\R\right) &=& \underbrace{\R\left\langle e^{123}+e^{345},\; e^{126}+e^{456} \right\rangle}_{=\; H^{(1,0)}_\omega\left(X;\R\right) \;=\; L\,H^{(0,1)}_\omega\left(X;\R\right)} \oplus \underbrace{\R\left\langle e^{123}-e^{345},\; e^{126}-e^{456} \right\rangle}_{=\; H^{(0,3)}_\omega\left(X;\R\right)}  \;, \\[5pt]
H_{dR}^4\left(X;\R\right) &=& \underbrace{\R\left\langle e^{1236}+e^{1245}+e^{3456} \right\rangle}_{=\; H^{(2,0)}_\omega\left(X;\R\right)} \oplus \underbrace{\R\left\langle e^{1236}-e^{1245},\; e^{1236}-e^{3456} \right\rangle}_{=\; H^{(1,2)}_\omega\left(X;\R\right) \;=\; L\, H^{(0,2)}_\omega\left(X;\R\right)} \;, \\[5pt]
H_{dR}^5\left(X;\R\right) &=& \underbrace{\R\left\langle e^{12456},\; e^{12345} \right\rangle}_{=\; H^{(2,1)}_\omega\left(X;\R\right) \;=\; L^2\, H^{(0,1)}_\omega\left(X;\R\right)} \;, \\[5pt]
\end{eqnarray*}
hence we get a decomposition
$$ H^\bullet\left(X;\R\right) \;=\; \bigoplus_{r\in\N} L^r\, H^{(0,\bullet-2r)}_\omega\left(X;\R\right) \;.$$
In particular, it follows that the Hard Lefschetz Condition holds on $\left(X,\, \omega\right)$.
\end{ex}

\begin{ex}
Take the $6$-dimensional completely-solvable solvmanifold
$$ X \;:=\; \left(-23,\; 0,\; 0,\; -46,\; 56,\; 0 \right) $$
endowed with the linear symplectic structure
$$ \omega \;:=\; e^{12}+e^{36}+e^{45} \;.$$

By Hattori's theorem \cite[Corollary 4.2]{hattori}, one computes
\begin{eqnarray*}
H^1_{dR}(X;\R) &=& \underbrace{\R\left\langle e^2,\; e^3,\; e^6\right\rangle}_{=H^{(0,1)}_\omega(X;\R)} \;,\\[5pt]
H^2_{dR}(X;\R) &=& \underbrace{\R\left\langle e^{12}+e^{36}+e^{45} \right\rangle}_{=H^{(1,0)}_\omega(X;\R)} \oplus \underbrace{\R\left\langle e^{12}-e^{36},\; e^{12}-e^{45},\; e^{13},\; e^{26} \right\rangle}_{=H^{(0,2)}_\omega(X;\R)} \;, \\[5pt]
H^3_{dR}(X;\R) &=& \R\left\langle e^{123},\; e^{126},\; e^{136},\; e^{245},\; e^{345},\; e^{456} \right\rangle \;. \\[5pt]
\end{eqnarray*}

Note that, being $e^{245}+\de e^{16}$ primitive,
$$ H^{(0,3)}_\omega\left(X;\R\right) \;\supseteq\; \R\left\langle e^{123}-e^{345},\; e^{126}-e^{456},\; e^{245} \right\rangle \;, $$
and, being $e^{245}-\de e^{16}=L\, e^2$,
$$ H^{(1,1)}_\omega\left(X;\R\right) \;=\; L\, H^{(0,3)}_\omega\left(X;\R\right) \;\supseteq \; \R\left\langle e^{123}+e^{345},\; e^{126}+e^{456},\; e^{245} \right\rangle \;, $$
while, being
$$ e^{136} \;=\; \underbrace{\frac{1}{2}\,\left(e^{136}+e^{145}\right)}_{\in L\,\Prim^1X} + \underbrace{\frac{1}{2}\,\left(e^{136}-e^{145}\right)}_{\in \Prim^3X} $$
and
$$ \de\wedge^2\duale{\mathfrak{g}} \;=\; \R \left\langle e^{146}-e^{234},\; e^{156}+e^{235},\; e^{236},\; e^{246},\; e^{256},\; e^{346},\; e^{356} \right\rangle \:, $$
one has
$$ \left\langle e^{136} \right\rangle \not\in H^{(0,3)}_\omega\left(X;\R\right) + H^{(1,1)}_\omega\left(X;\R\right) \;,$$
hence $H^{(0,3)}_\omega\left(X;\R\right) + H^{(1,1)}_\omega\left(X;\R\right) \subsetneq H^3_{dR}(X;\R)$.
\end{ex}

\medskip

The next example gives explicit examples of dual currents on a closed symplectic half-flat manifold.
\begin{ex}
Let $\C^3$ be endowed with the product $*$ defined by
$$
\left(w^1,\,
w^2,\,
w^3\right)*\,
\left(z^1,\,
z^2,\,
z^3\right)
=\,
\left(z^1+w^1,\,
\hbox{\rm e}^{w^1}z^2+w^2,\,
\hbox{\rm e}^{w^1}z^3+w^3\right)
$$
for every $\left(w^1, \, w^2, \, w^3\right),\, \left(z^1, \, z^2, \, z^3\right)\in \C^3$.
Then $\left(\C^3,\,*\right)$ is a complex solvable (non-nilpotent) Lie group and, according to \cite{Nak}, it admits lattice
$\Gamma\subset\C^3$, such that $X=\Gamma\backslash (\C^3,*)$ is a solvmanifold. Setting
$$
\varphi^1 \;:=\; \de z^1\;, \qquad \varphi^2 \;:=\; \hbox{\rm e}^{z^1}\de z^2\;, \qquad \varphi^3 \;:=\; \hbox{\rm e}^{-z^1}\de z^3\;,
$$
then $\left\{\varphi^1,\, \varphi^2,\, \varphi^3\right\}$ is a global complex co-frame on $X$ satisfying the following complex structure equations:
$$
\de\varphi^1 \;=\; 0\;,\quad \de\varphi^2 \;=\; \varphi^{12}\;,\quad \de\varphi^3\;=\;-\varphi^{13}\;.
$$
If we set $\varphi^j =: e^j+\im e^{3+j}$, for $j\in\{1,\, 2,\, 3\}$, then the last equations yield to
\begin{equation}\label{differenzialireali}
\left\{
\begin{array}{l}
\de e^{1} \;=\; \de e^{4} \;=\; 0\\[5pt]
\de e^{2} \;=\; e^{12}-e^{45}\\[5pt]
\de e^{3} \;=\; -e^{13}+e^{46}\\[5pt]
\de e^{5} \;=\; e^{15}-e^{24}\\[5pt]
\de e^{6} \;=\; -e^{16}+e^{34} \;.
\end{array}
\right.
\end{equation}
Then, (see \cite{dBT}),
$$
\omega \;:=\; e^{14}+e^{35}+e^{62}\;,
$$
and
$$
\begin{array}{lll}
Je^1 \;:=\; -e^{4}\;, & \quad Je^{3} \;:=\; -e^{5}\;, & \quad Je^{6} \;:=\; -e^{2}\;,\\[5pt]
Je^4 \;:=\;  e^{1}\;, & \quad Je^{5} \;:=\;  e^{3}\;, & \quad Je^{2} \;:=\; e^{6}\;,
\end{array}
$$
and
$$
\psi \;:=\; \left(e^1+\im e^4\right) \wedge \left(e^3+\im e^5\right) \wedge \left(e^6+\im e^2\right)
$$
give rise to a symplectic half-flat structure on $X$, where 
$$
\Re\mathfrak{e}\,\psi \;=\; e^{136} + e^{125} + e^{234} - e^{456} \;.
$$
Note that the Hard Lefschetz Condition holds on $\left(X,\, \omega\right)$, see \cite[Theorem 5.1]{dBT}.\\
Then, setting $z^j=:x^j+\im y^j$, for $j\in\{1,\, 2,\, 3\}$, and denoting by $\pi\colon \C^3\to X$ the canonical projection, we easily check that 
$$
\begin{array}{l}
Y_1 \;:=\; \pi\left(\left\{\left(x^1,\, x^2,\, x^3,\, y^1,\, y^2,\, y^3\right) \in \C^3 \st x^2=y^4=y^5=0\right\}\right) \;,\\[5pt]
Y_2 \;:=\; \pi\left(\left\{\left(x^1,\, x^2,\, x^3,\, y^1,\, y^2,\, y^3\right) \in \C^3 \st x^3=y^4=y^6=0\right\}\right)
\end{array}
$$
are special Lagrangian submanifolds of $\left(X,\, \omega,\, \psi\right)$, namely, for $j\in\{1,\, 2\}$, it holds $\Re\mathfrak{e}\,\psi\lfloor_{{Y_j}}=\Vol_{Y_j}$, and, consequently, the associated dual currents $\rho_{Y_j}$ are primitive.
\end{ex}

\end{document}